\documentclass{amsart}

\usepackage{amssymb}
\usepackage{amsmath}
\usepackage{amsthm}
\usepackage{amsbsy}
\usepackage{amscd}
\usepackage{bm}
\usepackage{graphicx}
\usepackage{tikz-cd}
\usepackage{enumerate}
\usepackage{mathtools}
\usepackage{comment}
\theoremstyle{plain}
\newtheorem{theorem}{Theorem}[section]
\newtheorem{lemma}[theorem]{Lemma}

\newtheorem{proposition}[theorem]{Proposition}

\theoremstyle{definition}
\newtheorem{definition}[theorem]{Definition}

\theoremstyle{remark}
\newtheorem{remark}[theorem]{Remark}

\newtheorem*{properties}{Properties}

\newcommand{\C}{\mathbb{C}}

\begin{document} 

\title{Symplectic embeddings of $4$--manifolds
via Lefschetz fibrations}


%

\author{Dishant M. Pancholi}
\address{ The Institute of Mathematical Sciences\\
           No. 8, CIT Campus, Taramani, \\
 	   Chennai 600113, India}
 \email{dishant@math.tifr.res.in}
 \author{Fransico Presas}
 \address{ICMAT, Madrid, Spain}
 \email{fpresas@icmat.es}

\date{\today}

\subjclass{Primary ; Secondary }

\begin{abstract}
In this article we study proper  symplectic and iso-symplectic  embeddings of $4$--manifolds in $6$--manifolds. We show that a closed orientable smooth $4$--manifold admitting a Lefschetz fibration  over $\C P^1$ 
admits a symplectic embedding in the symplectic manifold $(\C P^1 \times 
\C P^1 \times \C P^1, \omega_{pr}),$ where $\omega_{pr}$ is the product 
symplectic form on $\C P^1 \times \C P^1 \times \C P^1.$ We also show that there exists a sub-critical Weinstein $6$--manifold  in which all 
finite type Weinstein $4$--manifolds  admit iso-symplectic embeddings. 

\end{abstract}
\maketitle

\section{Introduction}\label{sec:introduction}
The study of embeddings of manifolds has a long and 
fascinating  history. Many important techniques
essential  for the study of geometric topology 
originated from the study embeddings of manifolds. 
To name a few, (1) H. Whiteny's famous trick~\cite{Wh} was discovered by him to establish 
embeddings of $n$--dimensional manifolds in 
$\mathbb{R}^{2n},$ (2)  Nash's $C^1$--isometric 
embedding theorem~\cite{Na} -- which can be regarded as precursor to the phenomenon of $h$--principle
discovered and popularized by M. Gromov~\cite{Gr}, and (3) 
Kodiara's embedding theorem~\cite{Ko}  which characterizes which complex manifolds are projective. 

In this article we discuss proper symplectic and iso-symplectic embeddings of manifolds. 
Recall that by a symplectic embedding of an even dimensional manifold $M$ in a  symplectic manifold $(W, \omega_W),$
we mean a proper smooth embedding of $M$ in $W$ such that
the pull-back of the symplectic form $\omega_W$ via the 
embedding induces a symplectic structure on $M.$  In case
$M$ is symplectic with symplectic form $\omega_M,$ and the pull back $f^*(\omega_W)$ via an embedding $f$ of $M$ in $W$ is the form $\omega_M,$ then we say that $f$ is an iso-symplectic embedding of $(M, \omega_M)$
in $(W, \omega_W).$

The study of symplectic and iso-symplectic  embeddings of  symplectic manifolds in a given symplectic manifold was initiated by M. Gromov.
Gromov showed that the question of iso-symplectic and symplectic embeddings
of symplectic manifolds  abides by the $h$-principle provided the co-dimension of the embedding is at least $4.$   In case, $(M, \omega_M)$ is an open manifold,  Gromov showed
that the iso-symplectic and symplectic embedding problem abides by the
$h$--principle even in the case when the  co-dimension of embeddings under consideration is $2.$ Apart form these earlier works of Gromov, a
major break through in producing symplectic sub-manifolds of a closed symplectic manifold 
came through the works of Donaldson~\cite{Do1}.
Donaldson, using his approximately holomorphic technique, 
provided many symplectic sub-manifolds of a given
closed symplectic manifold. In particular, he 
produce closed symplectic sub-manifolds in co-dimension $2.$ Donaldson's technique, though very powerful and extremely useful, does not provide any insight into the question finding which manifold might occur as a co-dimension $2$ symplectic sub-manifold of a given symplectic manifold.

In the present article we try to address this question. We will  focus on   
symplectic  embeddings  of closed $4$--manifolds in co-dimension $2,$ and    proper symplectic as well as  iso-symplectic embeddings of  Weinstein manifolds in co-dimension $2.$  

Let us now proceed towards stating  precisely  statements of  main results. It is well known that a large class of symplectic manifolds admit \emph{Lefschetz fibrations}. In this article 
we will  discuss symplectic and iso-symplectic embedding of manifolds admitting Lefschetz fibration. We 
 will also mostly focus on embeddings of $4$--manifolds 
 in $6$--dimensional symplectic manifolds occasionally pointing out the places where  there is a possibility of generalisations   to higher dimensions. 
 We begin by first discussing statements regarding symplectic embeddings of closed manifolds.

\subsection{Embeddings of closed manifolds}\mbox{}
 
Let us  recall few notions related to Lefschetz fibrations. 

 \begin{definition}[Lefschetz fibration]\label{def:lef_fib} 
Let $M$ be an oriented  $2n$--dimensional compact  manifold,  possibly with a non-empty 
boundary and corners.  By a Lefschetz fibration on $M,$ we mean a map $f: M \rightarrow \mathcal{S},$
where $\mathcal{S}$ is either $\C P^1$ or a  closed unit $2$--disk  
$\mathbb{D}^2,$ which satisfies the following property.

For every  $x$ at which  the map $f$ is singular, there exists an orientation preserving
parameterization  $\phi: U \subset M \rightarrow \mathbb{C}^n,$ and an orientation preserving 
parameterization $\psi: V \subset \mathcal{S} \rightarrow \mathbb{C}$ such that
the following properties are satisfied:

\begin{enumerate}
\item $x \in U,$ and $\phi(x) = (0,\cdots, 0) \in \mathbb{C}^n.$
\item $f(x) \in V,$ and $\psi(f(x)) = 0 \in \mathbb{C}.$
\item  For the map $g:\mathbb{C}^n \rightarrow \mathbb{C}$ given by 
$g(z_1, z_2, \cdots , z_n) = z_1^2 +z_2^2 + \cdots + z_n^2,$ the following diagram commutes:

\begin{center}

\begin{tikzcd}
 U \subset M \arrow[r, "\phi"] \arrow[d, "f"] &  \mathbb{C}^2  \arrow[d, "g"]  \\
V \subset \mathcal{S} \arrow[r,  "\psi" ] & \mathbb{C}. 
\end{tikzcd}
\end{center}
\end{enumerate}

\end{definition}

It is well known that given a Lefschetz fibration, a fiber containing a singular point is obtained from a nearby fiber $F$ by collapsing an embedded $n$--sphere $\mathbb{S}$ to a point. The sphere $\mathbb{S}$
is called a vanishing cycle. When $\mathbb{S}$ is homologically non-trivial, we say that $\mathbb{S}$ is an essential cycle.

 A Lefschetz fibration $\pi_M:M \rightarrow \Sigma,$  on $M$ which satisfies: 
(1) the map $\pi_M$ restricted to the set of critical 
points is injective, (2) every vanishing cycle is essential is also known as \emph{simplified Lefschetz fibration}~\cite{BO}. Most Lefschetz fibrations that one generally encounter satisfy these properties. Further, in most cases, given a manifold $M$ admitting a Lefschetz fibration which
is not simplified, it is possible to produce another Lefschetz fibration on $M$ which is simplified. For this reason, we will only focus on simplified Lefschetz fibration.  Hence from now on, by a  Lefschetz fibration, we will always mean a simplified Lefschetz fibration.

We use Lefschetz fibrations to produce embeddings,
for this we need to   recall the notion of Lefschetz fibration embedding. This notion was discussed in~\cite{GPan}.

\begin{definition}
 Let $\pi_M: M \rightarrow \mathcal{S}$ and $\pi_N:N \rightarrow \mathcal{S}$ be two Lefschetz fibrations. 
 An embedding $\phi: M \hookrightarrow N$ is said to be a Lefschetz fibration embedding of $M$ in $N,$ provided the following diagram commutes:

\begin{center}

\begin{tikzcd}
 M \arrow[r, "\phi"] \arrow[d, "\pi_M"] &  N  \arrow[d, "\pi_N"]  \\
\mathcal{S}\arrow[r,  "Id" ] & \mathcal{S}. 
\end{tikzcd}
\end{center}

\end{definition}

We fix the following convention: 

For the ease of 
notations in this article
article we will  denote by $\mathcal{P}$ the manifold
$\C P^1 \times \C P^1 \times \C P^1$ and by $\pi_i: 
\mathcal{P} \rightarrow \C P^1$ the projection on the
$i$'th factor, for $i \in \{1,2,3\}.$  On the manifold $\mathcal{P}$ there is a natural
symplectic form defined as $\pi_1^{*} (\omega_{FS}) + 
\pi_2^{*}(\omega_{FS}) + \pi_3^{*}(\omega_{FS}),$ where 
$\omega_{FS}$ is the standard Fubini-Study symplectic form on $\C P^1.$ Let us denote 
this symplectic form by $\omega_{pr}.$ The map 
$\pi_3: \mathcal{P} \rightarrow \C P^1$ is clearly trivial symplectic fibration 
having fiber symplectomorphic to the symplectic
manifold $(\C P^1 \times \C P^1, \pi_1^* \omega_{FS} + \pi_2^* \omega_{FS}).$ 
We will denote the symplectic manifold 
$\C P^1 \times \C P^1$ with symplectic form $\pi_1^* \omega_{FS} + \pi_2^* \omega_{FS}$ by $(\C P^1 \times \C P^1, \omega_{FS} + \omega_{FS}).$

\begin{theorem}\label{thm:symp_lef_embedding}
Let $M$ be a closed orientable $4$--manifold, and let $\pi_M: M \rightarrow \C P^1$ be a Lefschetz fibration. 
There exists a Lefschetz fibration embedding of $M$ in
$(\mathcal{P}, \omega_{pr})$ which satisfies the following:

\begin{enumerate}
 \item the embedding is symplectic,
 \item any smooth fiber of the fibration $\pi_M: M \rightarrow \C P^1$ is symplectic sub-manifold of a fiber of $\pi_3.$
\end{enumerate}

\end{theorem}

Since its introduction in symplectic geometry in  the seminal
article~\cite{Do} by S. K. Donaldson, the notion of Lefschetz fibration has become extremely 
important in symplectic topology. An almost immediate 
consequence of \cite[Theorem 2]{Do} --  which provides the existence of \emph{Lefschetz pencil} on symplectic manifolds -- and Theorem~\ref{thm:symp_lef_embedding} is the following:

 \begin{theorem}\label{thm:embed_upto_blowup}
  Let $(M,\omega)$ be a closed $4$--dimensional symplectic manifold. After a finite number of blow-up there exist a symplectic embedding of the blown-up manifold $\mathcal{B}M$ in symplectic manifold $(\C P^1 \times \C P^1 \times \C P^1, \omega_{pr}),$ where $\omega_{pr}$ is the
  product symplectic form on $\C P^1 \times \C P^1 \times \C P^1.$
 \end{theorem}

\subsection{Embeddings of Weinstein manifolds}\mbox{}
 
Next, we apply our method of studying embedding via 
Lefschetz fibration to the case of Weinstein 
manifolds. Recall that a Weinstein manifold $W$ is a manifold admitting a triple $(\omega, \phi, X),$ 
where $\omega$ is a symplectic structure on $W,$   $\phi$ is a exhausting Morse function for $W$,
and $X$ is a complete vector field on $W$ which is exhaustive for $\phi$ and Liouville for $\omega.$ The triple $(\omega, \phi, X)$ is known
as a Weinstein structure on $W.$
 In this article we will
only deal with Weinstein manifolds that  admit  exhaustive Morse function with only finite number of critical points. Recall that  such
Weinstein manifolds are known as  \emph{finite type} Weinstein manifolds. Finally recall that a Weinstein manifold $V^{2n}$
of  dimension $2n$
admitting a exhaustive Morse function having critical
points of index at most $\frac{n}{2} - 1$ is known
as a sub-critical Weinstein manifold. 

  Our main result
related to iso-symplectic embeddings of Weinstein manifolds is the following:

\begin{theorem}\label{thm:weinstein_embedding}
   Let $\mathcal{L}(-3)$ denote the complex line bundle over $\C P^1$ with Chern class $-3.$ The manifold $L(-3) \times \mathbb{C} $ admits a 
  Weinstein structure $(\omega_U,  X, \phi)$ which satisfies the following property. 
  
  If $(V, \phi, X)$ is a Weinstein manifold of dimension $4$ satisfying the property that $\phi$ has only finite number of critical points, then 
  there exists an  iso-symplectic embedding of $(V, X , \phi)$ in  $(L(-3) \times \mathbb{C}, \omega_U, \phi, X).$ 
\end{theorem}

We would like to point out that the universal model for iso-symplectic embeddings that we have constructed is not unique. A large class of
Stein manifolds can serve as universal model. See the remark after the proof of Theorem~\ref{thm:weinstein_embedding} at the end of Section~\ref{sec:stein_lef_fib_embedding}. Furthermore, since any Stein manifold has an underlying Weinstein structure, Theorem~\ref{thm:weinstein_embedding}  provides iso-symplectic embeddings of Stein manifolds which are not holomorphic embeddings.

  It was pointed out to us by Prof. Yakov Eliashberg that a complete $h$--principle for iso-symplectic embedding of Weinstein manifold in co-dimension $2$ is well known provided the target manifold is flexible. This follows form the work discussed in \cite{CE}. However, producing an embedding of
  Weinstein manifold which is formally iso-symplectic is generally not very easy. We circumvent this problem by producing explicit iso-symplectic embeddings.

Let us have  few words regarding the arrangement of this article. Essential preliminaries related to Lefschetz fibration and mapping class groups are collected in  Section~\ref{sec:preliminaries}. In Section~\ref{sec:outline} we discuss main ideas involved
in proofs of Theorem~\ref{thm:symp_lef_embedding} and
Theorem~\ref{thm:weinstein_embedding}. The notion of 
\emph{flexible embeddings} of surfaces is discussed in
Section~\ref{sec:flexible_embedding} while sections~\ref{sec:proofs_lfe} and \ref{sec:stein_lef_fib_embedding} deal with proof of
Theorem~\ref{thm:symp_lef_embedding} and Theorem~\ref{thm:weinstein_embedding} respectively.

\subsection{Acknowledgments} We are  extremely thankful to Prof. Yakov Eliashberg for his constant encouragement, support, and critical comments. Dishant Pancholi is very thankful to T R Ramadas and Krishna Hanumantu for
various discussions related to this article. F. Presas is supported by the Spanish Research Projects SEV-2015-0554, MTM2016-79400-P, and MTM2015-72876-EXP.

\section{Preliminaries }\label{sec:preliminaries}

We start this section by recalling some results
related to Lefschetz fibrations.

\subsection{Lefschetz fibration}\mbox{}

 In this sub-section we recall two theorems about Lefschetz fibrations. First of which is due  S. K. Donaldson~\cite[Theorem:2]{Do} and
 J.Amor\'s, V. Mu\~noz  and F. Presas~\cite{AMP04},  which roughly says that every symplectic manifold, after finite number of blow-up admits a 
Lefschetz fibration structure. The second one is
due to R. Gompf which --in certain  sense -- establishes 
the converse of Donaldson's result providing the existence of symplectic structures on Lefschetz fibrations. 

\begin{theorem}[S. Donaldson~\cite{Do},  J. Amor\'os, V. Mu\~noz, and F.Presas~\cite{AMP04}]\label{thm:Donaldson}
 Let $(M, \omega)$ be a closed symplectic manifold. After
 finite number of blow-ups of $M,$ there exists a 
Lefschetz fibration structure on the blown-up manifold $\widetilde{M}$ which satisfies the property that regular fiber of this fibration is a symplectic sub-manifold of $\widetilde{M}.$  Furthermore, we can always find a Lefschetz fibration on $\widetilde{M}$ which is simplified.
\end{theorem}

We would like to mention here that Theorem~\ref{thm:Donaldson} is actually not stated in the form mentioned here but it is an easy consequence of
Theorem 2 of~\cite{Do}, and  \cite[Theorem 1.3]{AMP04}. Theorem 2 of \cite{Do} provides us with the existence of Lefschetz pencils, 
while  \cite[Theorem 1.3]{AMP04} shows   that
given a pair of vanishing cycles $c_1$ and $c_2$ 
associated to a symplectic Lefschetz fibration, there exists a symplectic  isotopy of the symplectic manifold which send $c_1$ to $c_2.$ This implies that even if one vanishing cycle is essential all vanishing cycles are essential.  This, in particular, implies that pencil induced by large degree line bundles satisfy
the property that all its vanishing cycles are essential. 

Finally let us  state a result due to R. Gompf~\cite[Theorem:10.2.18]{GS}:

\begin{theorem}[R.Gompf]\label{thm:gompf_existence_symp_form}
Let $X^{2n}$ be a smooth manifold that admits a Lefschetz
fibration $\pi: X \rightarrow \mathcal{S}$ with fiber $F.$ Then $X$ admits a symplectic structure on $\omega$ for which the  fiber $F$ is symplectic if and only if $[F]$ in  $H^2(X,\mathbb{Z})$ is non-zero. Furthermore, if $\{e_1,$ $e_2, \cdots e_n\}$ is a finite collection of sections of the Lefschetz fibration, then the symplectic structure  can be assumed to be such that each one of these sections is symplectic.
\end{theorem}

We would like to remark that we are not assuming that
$X$ is a closed manifold in the statement of Theorem~\ref{thm:gompf_existence_symp_form}. This will
be the case provided $\mathcal{S}$ is $\C P^1$. In
case, $\mathcal{S}$ is $\mathbb{D}^2,$ then $X$ 
is a manifold with boundary and corners.

\subsection{Review of symplectic mapping class group}\label{sec:mcg_preliminaries}
\mbox{}

One of the main ingredient for establishing symplectic embeddings 
is the notion of flexible embedding of surfaces. In order to get
flexible embeddings of surfaces, we use some results about mapping class groups of surfaces. Let us review these.

Throughout this article by the mapping class group of
an orientable surface $(\Sigma, \omega)$ we mean the group of
symplectic form  preserving diffeomorphism of $\Sigma$ up to a symplectic isotopy. In case, the boundary of
$\Sigma$ is non-empty, then the mapping class group
consists of all symplectomorphisms which are identity when restricted to the boundary of
$\Sigma$ up to isotopies that are identity when restricted
to the boundary of $\Sigma.$. Furthermore, since the symplectic isotopy class of a particular symplectomorphism is the only thing that
is relevant for this article, the word  symplectomorphism will always mean the isotopy class of the symplectomorphism.

It follows from the works of M. Dehn~\cite{De} and C. Lickorish~\cite{Li} that the mapping class group of a closed orientable 
surface is generated by Dehn twists~\cite[Section 3.1.1]{BM}. 
S. Humphries extended their work to  established the most economical set of generators for the mapping class group 
of an orientable genus $g$ surface. He showed that the mapping class group is
generated by Dehn twists along the curves $a_i,$ $i = 1$ to $g$, $b_j$, $j =1$ to $g_1,$
$c_1,$ and $c_2$ as depicted in Figure~\ref{fig:humphries_generators} provided, $\Sigma$ is a closed orientable  surface of genus $g, g\geq 2.$  

Since we are working with surfaces together with a symplectic form $\omega,$ $\Sigma$ naturally comes 
equipped with an orientation. When we say a positive  
(left handed) Dehn twist~\cite[Section 3.1.1]{BM}, we always mean a positive Dehn
twist with respect to this orientation. The general term
Dehn twist refers to either a positive or  a negative Dehn twist.

\begin{figure}[ht]\label{fig:humphries_generators}
 \begin{center}
\includegraphics[scale=0.7]{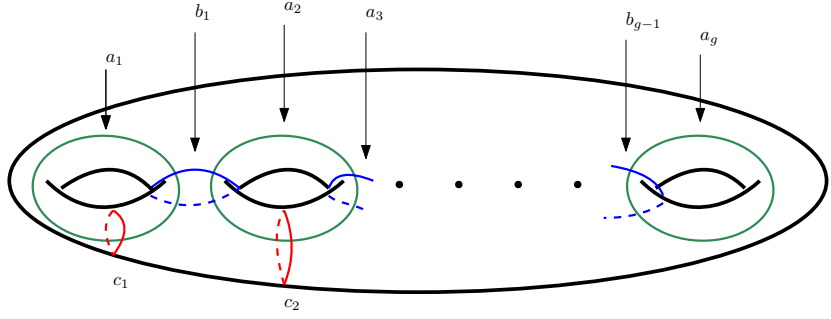}
 \end{center}
\caption{The picture  depicts the set of
Humphries generators for the mapping class group of a closed orientable surface of genus $g.$ The curves $a_i$'s are depicted in green, curves
$b_j$'s are depicted in blue, curves $c_1$ and $c_2$ are
depicted in red.}
\end{figure}

For a surface $(\Sigma, \partial \Sigma)$ having a unique
boundary component as depicted in Figure~\ref{fig:humphries_generators_with_boundary}, the mapping
class group is generated by the same set of Humphries generators together with Dehn twist along boundary parallel
curve $d$ as depicted in Figure~\ref{fig:humphries_generators_with_boundary}.

\begin{figure}[ht]\label{fig:humphries_generators_with_boundary}
 \begin{center}
\includegraphics[scale=0.7]{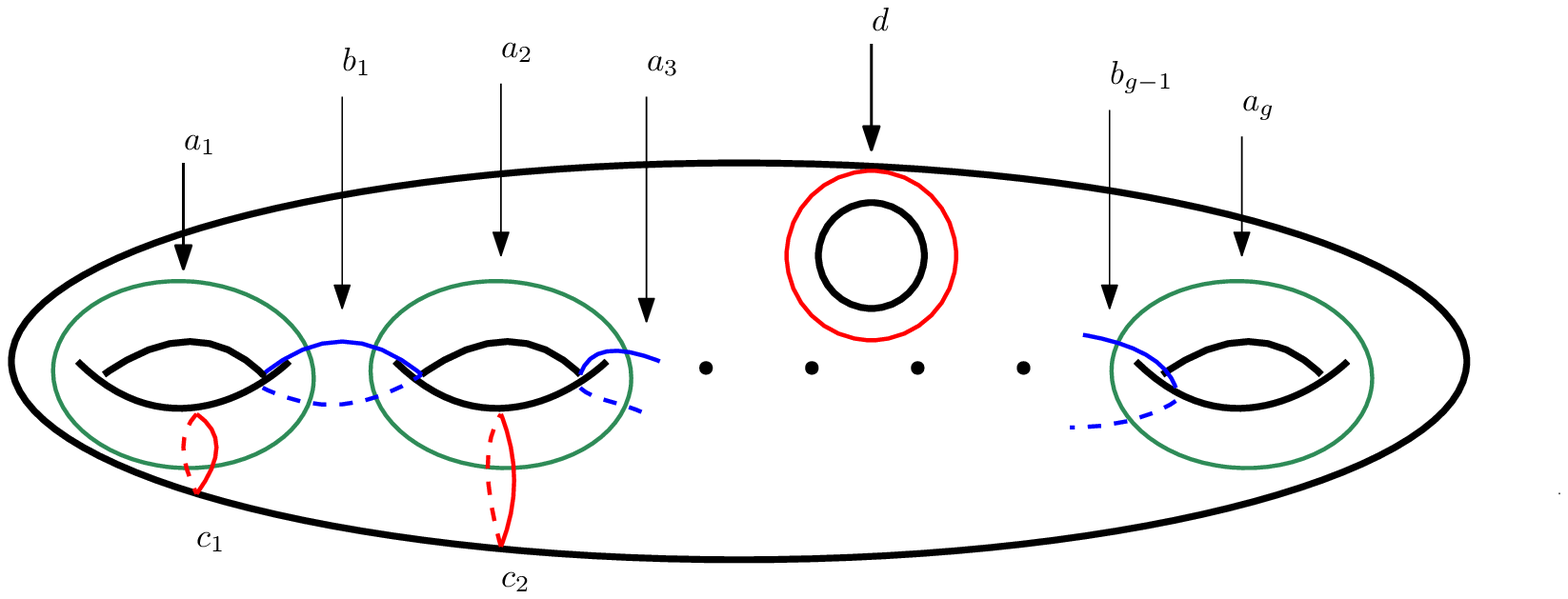}
 \end{center}
\caption{The figure depicts Humphries generators for the mapping class group  of a surface with single boundary component and a surface having exactly one puncture. The mapping class group of punctured surface is generated 
by Dehn twists along $a_i, b_j, c_1,$ and $c_2,$ while 
to generate the mapping class group of the surface
with one boundary component we need an  additional Dehn twist along the curve $d$.}
\end{figure}

Having collected necessary preliminaries, we proceed now
to provide proofs of main results. To make our ideas
accessible, in the next section, we  outline the main ideas
involved in the proofs.

\section{Main ideas involved in proofs
of Theorem~\ref{thm:symp_lef_embedding} and Theorem~\ref{thm:weinstein_embedding}}.\label{sec:outline}

Let us first discuss ideas involved in establishing
Theorem~\ref{thm:symp_lef_embedding}. Given any Lefschetz
fibration $\pi_M:M \rightarrow \C P^1,$ we know that
removing finite number of  singular fibers, we get
a fiber bundle over punctured $\C P^1.$ We also know that
the monodromy around each of this puncture is a positive
Dehn twist. 

Let $F$ denote a fixed smooth fiber of
the Lefschetz fibration $\pi_M:M \rightarrow \C P^1.$
We observe the following: if there exists a symplectic manifold $(W, \omega_W)$ and a symplectic embedding $i: F \hookrightarrow W$  such that
for every symplectomorphism $\psi: F \rightarrow F$ the
embeddings $\psi \circ i$ and $i$ are symplectically isotopic, then the fiber bundle over punctured $\C P^1$
admits a symplectic embedding in $W \times \C P^1.$ 
An embedding satisfying this property is termed as
\emph{symplectically flexible embeddings}. We refer to
Section~\ref{sec:flexible_embedding} for a precise 
definition of symplectically flexible embeddings.

We observe that $(\C P^1 \times \C P^1, \omega_{FS} + \omega_{FS})$ is the symplectic manifold in which 
every genus $g$ surface admits a symplectically flexible embedding. The existence of symplectically flexible
embeddings is discussed in Lemma~\ref{lem:symp_flex_embed_surface}. As remarked
earlier, Lemma~\ref{lem:symp_flex_embed_surface} implies
that there is fiber preserving symplectic embedding
of the Lefschetz fibration $\pi_M:M \rightarrow \C P^1$
restricted to the complement of singular fibers in 
the trivial fibration $\pi_3: \mathcal{P} \rightarrow \C P^1,$ where recall that $\mathcal{P} = \C P^1 \times \C P^1 \times \C P^1$ with product symplectic form and 
$\pi_3$ is the projection on the third factor. 

In order to extend the fiber preserving embedding constructed in previous paragraph to an embedding of
$M$ we turn to local model of Lefschetz critical point.
We observe that there exist a symplectic  embedding
$\phi$of  $\mathbb{C}^2$ in $\mathbb{C}^3$ given by $\phi(z_1, z_2) =
(z_1, z_2, 0)$ such that the following diagram commutes:

\begin{center}

\begin{tikzcd}
 \mathbb{C}^2 \arrow[r, "\phi"] \arrow[d, "f"] &  \mathbb{C}^3  \arrow[d, "g"]  \\
\mathbb{C} \arrow[r,  "id" ] & \mathbb{C}, 
\end{tikzcd}
\end{center}

\noindent where $f(z_1, z_2, z_3) = z_1z_2 + z_3$ and
$g = f \circ \phi .$ Observe that this implies we can 
get embed a neighborhood of Lefschetz singular point in a
standard symplectic ball such that trivial fibration of this
ball to a disk induces the given Lefschetz fibration 
on the neighborhood of Lefschetz singularity. This observation allows us to extend the embeddings obtained
at the end of first step to neighborhoods of singular fibers.  Though we have not written down final argument in
this explicit format, the main idea behind proofs is
this. We refer to~\cite{GPan} for smooth embeddings of all
$4$--manifolds in $\C P^3$ using similar ideas. In fact,
our main observation was that most of the ideas discussed
in~\cite{GPan} can be adopted in symplectic setting.

Let us end this section be briefly discussing main
ideas in the proof of Theorem~\ref{thm:weinstein_embedding}. Theorem~\ref{thm:weinstein_embedding} follows relatively
easily from Theorem~\ref{thm:stein_lef_fib_embedding} where we discuss Stein Lefschetz fibration embedding in 
a fixed Stein domain. 

The ideas involved in producing
Stein Lefschetz fibration embeddings are similar to the
one discussed earlier for the proof of Theorem~\ref{thm:symp_lef_embedding}. It is clear that
one needs to produce a Stein manifold $W$ which 
admits flexible embeddings of Stein $1$--manifolds. We observe that 
the Stein manifold $\mathbb{D}\mathcal{L}(-3)$ has the
property that every Stein $1$--manifold  which is biholomorphic to a  once punctured Riemann surface,  admit a  proper
symplectically flexible embedding in $\mathbb{D}\mathcal{L}(-3).$ See  Lemma~\ref{lem:flex_embedding_with_boundary} for 
a precise statement regarding flexible embeddings in
$\mathbb{D}\mathcal{L}(-3).$ Rest of the proof to establish
Theorem~\ref{thm:stein_lef_fib_embedding} follows
essentially the same logic.

\section{Flexible embeddings of surfaces}\label{sec:flexible_embedding}

We begin by recalling few definitions regarding 
flexible embeddings of surfaces from~\cite{GPan}.

\begin{definition}\label{def:conjugated_element_of_mcg}
 Let $\Sigma$ be an orientable surface. Let $(M,\omega)$ be a symplectic manifold. Let $\Psi: \Sigma_g \hookrightarrow (M, \omega)$ be a symplectic embedding. Let  $f$ be an element of 
 $\mathcal{M}CG(\Sigma, \Psi^{*}(\omega) = \Omega).$  We say that the element $f$  is conjugate via
 the embedding $\Psi$ provided the following properties are satisfied:
 
 \begin{enumerate}
  \item There exits a $1$--parametric family $\phi_t$ of symplectomorphisms of $(M, \omega)$ such that $\phi_0 = id$ and $\phi_1 (\Psi(\Sigma)) = \Sigma.$
  \item $\Psi^{-1} \circ \phi_1 \circ \Psi = f. $
 \end{enumerate}

 \end{definition}

 \begin{definition}[Symplectically flexible embedding]\label{def:flexible_embedding}
  A symplectic  embedding $\Psi$ of an orientable surface $\Sigma$ in a symplectic manifold $(M, \omega)$ is said to be symplectically flexible provided every $f \in \mathcal{M}CG(\Sigma, \Psi^* (\omega))$ is conjugated 
  by a symplectic isotopy of $(M, \omega).$
 \end{definition}

For the sake of brevity, we will refer to a symplectically flexible embedding just by the term \emph{flexible
embedding}. Recall that the term flexible embedding is used in \cite{GP} for smooth embedding 
satisfying properties similar to the define in Definition~\ref{def:flexible_embedding}. Since in this article,
we will only be dealing with symplectically flexible embeddings, we are going to take   slight liberty and refer them as flexible embeddings in the rest of the article unless stated otherwise explicitely. 
 
 Consider a pencil $\pi:M^4 \setminus B \rightarrow \C P^1$
 on a closed symplectic manifold $(M, \omega)$. Let $c$ be a critical
 point for the pencil and let $\nu$ be a vanishing 
 cycle corresponding to the critical point $c.$ Our first
 lemma is regarding the existence of a symplectic isotopy
 of $M$ which conjugates the element $\tau_{\nu},$
 where $\tau_{\nu}$ denotes the positive Dehn twist along
 the curve $\nu.$

 \begin{lemma}\label{lem:ambient_twist_along_vanishing_cycle}
 Let $p:M \setminus B \rightarrow \C P^1,$ $c,$ $\nu,$ and
 $\tau_{\nu}$ be as in the previous paragraph. Let $F$
 be a fiber of pencil with $\nu \subset F.$ Then there
 exists a $1$-parametric family, $\Psi_t,t \in [0,1],$ of symplectomorphisms of $M$  and a $4$--ball 
 embedded in $M$ which satisfies the following:
 
 \begin{enumerate} 
  \item The $4$--ball is symplectomorphic to standard 
  symplectic $4$--ball of some radius $r.$
  \item For each $t,$ $\Psi_t$ is the identity outside
  $B^4$ and $\Psi_0$ is the identity.
  \item The family $\Psi_t$ conjugates $\tau_{\nu}.$
 \end{enumerate}
\end{lemma}

 \begin{proof}
In order to prove this result, let us first consider an abstract Weinstein Lefschetz fibration $\pi_B:B^4 \rightarrow \mathbb{D}^2,$ where $\mathbb{D}^2$ is 
the closed  disk of radius $\delta$ around $0$ in 
$\mathbb{C},$ which satisfies the following properties:

\begin{enumerate}
 \item The fibration has a unique Lefschetz critical point on the fiber over $0,$ and each 
fiber symplectomorphic to a Weinstein domain $D T^* \mathbb{S}^1,$ where $D T^* \mathbb{S}^1$ denotes a disk bundle associated to the co-tangent bundle $T^* S^1.$

\item The  fibrations  is such that 
the monodromy associated  to the fiber bundle $\pi_B: B^4 \setminus \pi_B^{-1}(\{0\}) \rightarrow 
\mathbb{D}^2 \setminus \{0\}$ is a Dehn twist which
is supported away from a neighborhood of $\partial D 
T^* \mathbb{S}^1.$

\end{enumerate}
\
Consider a  vector-field  $V$ on $\mathbb{D}$ such that
the flow associated to this vector field keeps a small
collar of the boundary $\partial \mathbb{D}$ fixed and rotates the circle of radius less than or equal to $\frac{1}{4} \delta$ by $2\pi$ while rotating the 
circle of radius $\frac{1}{2} \delta$ by $\pi$ as depicted in Figure~\ref{fig:vector_field_on_disk}. The lift of this 
vector field via a symplectic connection on the fibration
produces a  flow that conjugates the monodromy  Dehn twist on the fiber $\pi_B^{-1}\{(\frac{1}{2},0)\}$
Let us denote this flow by $\widetilde{\phi_t}.$ 

Next, construct a properly  embedding  of an annulus $A$ in $B$ which satisfies the following:

\begin{enumerate}
 \item $A = \widetilde{D} T^* \mathbb{S}^1 \subset 
 D T^* \mathbb{S}^1,$ where 
 $\widetilde{D}$ is disk of smaller radius than the
 radius of the disk $D,$ and the disk $\widetilde{D}$ is chosen such that the monodromy Dehn twist 
 associated to the fibration $\pi_B: B \rightarrow \mathbb{D}$ is the identity when restricted to 
 $D T^* \mathbb{S}^1 \setminus \widetilde{D} T^* \mathbb{S}^1.$
 
\item $\partial A =  \pi_B^{-1}(\{0\}) \cap \partial B^4.$
\item The embedding is symplectic.
\end{enumerate}

Observe that the flow $\widetilde{\phi_t}$ conjugates 
the Dehn
twist on $A,$  This implies by enlarging
the vertical boundary of the Weinstein fibration 
$\pi_B: B^4 \rightarrow \mathbb{D},$ we get that
there exists a Weinstein fibration $\pi: \widetilde
{B} \rightarrow \mathbb{D}^2,$ and an embedding of
an annlus $\widetilde{A}$ in $\widetilde{B}$ which satisfies the
following:

\begin{enumerate}
 \item $A \subset \widetilde{A}.$
 \item There exist a family $\widetilde{\phi}_t$
 of diffeomorphism  which conjugates a Dehn twist
 on $\widetilde{A}$
 \item For $\widetilde{\phi}_t= \phi_t,$ when restricted to $B \subset \widetilde{B}.$ 
\end{enumerate}

Till now, we have produce the flow $\widetilde{\phi_t}$
 which need not be a symplectic flow,  however, 
 $\widetilde{\phi_t}^{-1}(\{\frac{1}{2},0\})\}$ is a $1$--parametric  family of symplectic sub-manifolds of $B^4$ all of them
have common boundary embedded in $\partial B^4.$ Hence in light of~\cite{Ar}[Proposition:4] by D. Auroux, we get that there exits a flow
 $\phi_t$ which  preserve the symplectic structure 
 associated to the Weinstein fibration -- i.e., $\phi_t$ is a symplectomorphism
for each $t$ -- $\pi_B: B \rightarrow \mathbb{D},$ and agrees with $\widetilde{\phi_t}$ near the boundary of $B^4.$
 Clearly, $\phi_t$ Observe that the pair $(\widetilde{A}, \phi_t)$ satisfy the following properties:
\begin{enumerate}
 \item $\widetilde{A}$ is a symplectically embedded annulus
 in $B^4$
 symplectomorphic to a unit disk bundle associated to $T^* \mathbb{S}^1.$
 
\item $\phi_t$ is a $1$--parametric family of symplectomorphism of $B^4$ each identity close to $\partial B^4$ and which conjugates the Dehn twist on $\widetilde{A}.$ 
\end{enumerate}

Now, given $p:M \setminus B \rightarrow \mathbb{C} P^1,c, \nu,$ and $\tau_{\nu},$ without loss of 
generality we can assume that $\tau_{\nu}$ is supported
in a small neighborhood of $\nu.$ This implies that
there exists a neighborhood $\mathcal{N}$ of the critical point
$c$ symplectomorphic to some abstract Weinstein 
fibration of the type $\pi_B: B \rightarrow \mathbb{D}^2.$ This clearly implies that there exits a 
family $\widetilde{\Psi}_t$ of symplectomorphisms of supported
in $\mathcal{N}$ which conjugates positive Dehn
twist on an annulus $\mathcal{A}$ embedded in $\mathcal{N}$
which satisfies the property that $\partial A$
is pair of circles on the singular fiber $\pi_M^{-1}
(\{c \}).$ 

Finally, consider the symplectic sub-manifold $\widetilde{F}$  of $M$ which is the union
of $\mathcal{A}$ with $\pi_M^{-1}(\{c\}) \setminus \mathcal{N}.$
Clearly, this is a symplectically embedded 
sub-manifold symplectically isotopic to any smooth fiber $F$ of the pencil $\pi_M: M \setminus B \rightarrow \C P^1.$  Hence, we get that given
a fiber $F$ and $\nu,$ on the fiber $F$ there exists
a family $\Psi_t$ of symplectomorphisms of $M$
which conjugates $\tau_{\nu}.$ Hence the proof.

\end{proof}

\begin{figure}[ht]\label{fig:vector_field_on_disk}
 \begin{center}
\includegraphics[scale=0.8]{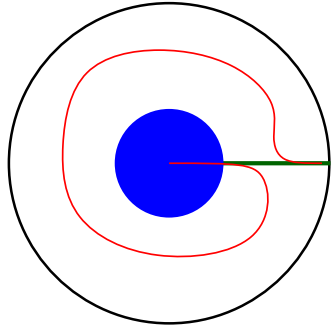}
 \end{center}
\caption{The picture depict the flow assoicated 
to the vector field whose
lift via symplectic connection produces the flow on
$B^4$ which induces positive Dehn twist on the 
fiber over $\frac{1}{2} \delta + i 0.$ The vector field is
such that the associated flow rotates the blue disk  by $2\pi$ while keeping the boundary fixed. The red curve depicts
the time $1$  image of the black curve under the flow. }
\end{figure}
 
\subsection{Flexible embeddings in $\mathbb{C}P^1 \times \mathbb{C}P^1$}\mbox{}

 \begin{lemma}\label{lem:symp_flex_embed_surface} Consider the symplectic manifold $(\C P^1 \times \C P^1, \omega_{FS} + \omega_{FS})$. For every $g$ positive there exists a flexible symplectic embedding of genus  $g$ surface in $(\C P^1 \times \C P^1, \omega_{FS} + \omega_{FS}).$
 \end{lemma}
 \begin{proof}
 Recall that $\C P^1 \times \C P^1$ admits holomorphic Lefschetz  pencil having its smooth fiber a  
  symplectically  embedded $2$--torus $\mathbb{T}^2,$ and the base locus  consisting of $8$ points. Furthermore, the $2$--torus is the zero of a generic 
 section of the line bundle  $\mathcal{O}(2) \otimes \mathcal{O}(2),$ and the vanishing cycles -- up to Hurewicz moves -- consist of curves $a$ and $b$ as depicted
  in Figure~\ref{fig:basic_flex_torus} having relation
  $\left(\tau_{a} \cdot \tau_{b} \right)^6 = Id.$ Let us denote this pencil by $\pi_{(2,2)}: \C P^1 \times \C P^1  \setminus B_{(2,2)} \rightarrow \C P^1,$ where $B_{(2,2)}$
  denotes the base locus for the pencil consisting of $8$
  points.
  
 Since the homology class of this torus is $(2,2),$ we know that any such torus is symplectically isotopic to  a symplectically  embedded torus constructed via the following procedure. 
 
 Choose a pair of points $N, S$ in $\C P^1.$ Next, consider two vertical
 $\C P^1$ given by $\{N\} \times \C P^1$ and $\{S \} \times
 \C P^1.$ Next, consider two horizontal $\C P^1$'s given
 by $\C P^1 \times \{N\}$ and $\C P^1 \times \{S\}.$ Now 
 perform ambient Gromov sum at the points $\{(N,N),(N,S)
 (S, N), (S,S)\}.$ Let us denote this torus by $\mathbb{T}_{(2,2)}.$
 
 We claim that $\mathbb{T}_{(2,2)}$ is flexible in $(\C P^1
 \times \C P^1, \omega_{FS} + \omega_{FS}).$ This is because, the pencil $\pi_{(2,2)}$  on $\C P^1 \times \C P^1$ with fiber 
 $\mathbb{T}_{(2,2)}$ has vanishing cycles isotopic to
 the curves $a$ and $b$ as depicted in the Figure~\ref{fig:basic_flex_torus}. Hence according to 
 Lemma~\ref{lem:ambient_twist_along_vanishing_cycle}
we can conjugate Dehn twists $\tau_{a}$ and $\tau_{b}$ via family symplectomorphisms. Furthermore, each one of this family is 
supported in a small ball which contains a tubular
neighborhood of the vanishing cycle. We know that $\tau_{a}$ and $\tau_{b}$ generate the mapping class group of any torus, 
and hence we get that $\mathbb{T}_{(2,2)}$ is flexible in $\C P^1 \times \C P^1.$

\begin{figure}[ht]\label{fig:basic_flex_torus}
 \begin{center}
\includegraphics[scale=0.7]{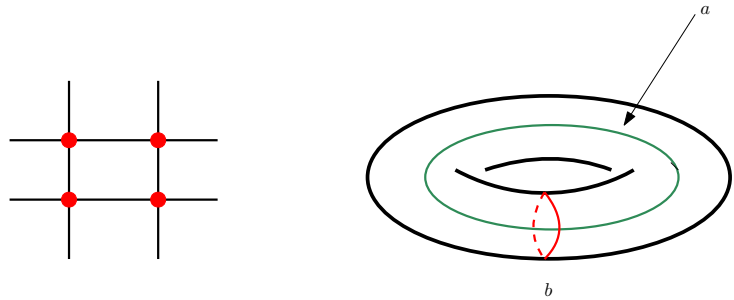}
 \end{center}
\caption{ The left of the figure is the schematic description of the embedded curve $\mathbb{T}_{(2,2)}$. 
Two vertical lines correspond to sphere $\{N\} \times \C P^1$ and $\{S\} \times \C P^1,$ while the two horizontal
lines depict $\C P^1 \times \{N\}$ and $\C P^1 \times \{S\}$. The red disk depicts the Gromov sum along the point of
intersection. On the right of the figure the picture depicts the curve $\mathbb{T}_{(2,2)}$ obtained as a result Gromov
sum of  two vertical $\C P^1$'s  with two horizontal $\C P^1$'s together with vanishing cycles for the pencil $\pi_{(2,2)}.$ }
\end{figure}

Till now we have shown that there exists a flexible embedding of tours in $\C P^1 \times \C P^1.$ We now
show how to produce a symplectic embedding of a surface of genus $g$ for any $g>0.$ Consider $g-1$ distinct points $P_1, \cdots, P_{g-1}$
on $\C P^1$ such that $\{P_1, \cdots, P_{g_1}\} \cap
\{N, S\} = \emptyset.$ Furthermore, we choose the 
points so that
non of the point in the collections $\{(N, P_i)\}$ or
$\{(S, P_i)\}$ lie on curves $\lambda$ and $\mu$ embedded
in $\mathbb{T}_{(2,2)} \subset \C P^1 \times \C P^1.$

Next, consider the  vertical $\C P^1 \times \{P_i\}$
which intersect the torus $\mathbb{T}_{(2,2)}$ in a 
pair of points $(N, P_i)$ and $( S, P_i).$ For each $i$
perform ambient Gromov sum to produce an embedding
of genus $g$ surface. We claim that this embedded surface
is flexible.  Let us denote this embedded surface by 
$\Sigma_g$

In order to establish the claim, notice that $\Sigma_g$
for each $g$ admits an embedding of a torus isotopic to 
$\mathbb{T}_{(2,2)}$ minus a pair of disk. This torus consist of the tours obtained by taking ambient Gromov sum of $\C P^1 \times \{P_i\},$ $\C P^1 \times \{P_{i+1}\}$ for $ i = 0, \cdots, g+1,$ where $P_0 = N$ and $P_{g+1} = S$ with
$\{N\} \times C P^1$ and $\{S\} \times \C P^1$ along 
four points of intersections. Let $a_i$ and $b_i$ be curves of this torus which corresponds to vanishing
cycles for the pencil $\pi_{(2,2)}.$ Hence, by
Lemma~\ref{lem:ambient_twist_along_vanishing_cycle} for the embedded
surface $\Sigma_g,$ Dehn twists $\{\tau_{a_i}, \tau_{b_i}\}, i =1, \cdots, g$ can be conjugated in
$\C P^1 \times \C P^1$ via families of symplectomorphisms each family having support 
in a small ball. Furthermore, notice that since the
relation for the pencil among Dehn twists is 
$$\left( \tau_{a_i} \cdot \tau_{b_i}\right)^6 = 1,$$

\noindent we get that for the curve $c$ depicted 
in Figure~\ref{fig:surface_flex} the Dehn twist $\tau_{c}$ also can be
conjugated. Now we know form Theorem~\cite{BM} that
the mapping class group of genus $g$ surface is
generated by Dehn twists along curves $a_i, b_j,c_1$ and $c_2$ and 
the collection of curves on which we can perform Dehn twist in 
the embedded surfaces  contains these curves. Hence we have the lemma.
 
 \end{proof}

\begin{figure}[ht]\label{fig:surface_flex}
 \begin{center}
\includegraphics[scale=0.6]{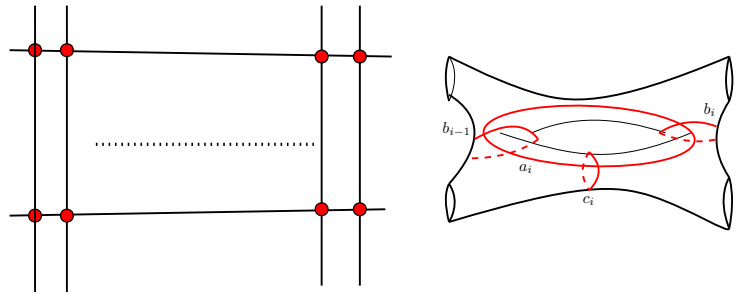}
 \end{center}
\caption{The figure depicts  a flexibly  embedded surface in
$\C P^1 \times \C P^1$ which is constructed out of a flexibly embedded
torus 
$\mathbb{T}_{(2,2)}.$ As usual horizontal lines depict $\C P^1 \times \{p\}$,
where $p$ is a point in $\C P^1,$ while vertical lies depict $\{q\} \times
\C P^1,$ where $q \in \C P^1.$ The red dot at the intersection depicts 
that we have performed ambient symplectic Gompf sum in a neighborhood of
an intersection point of a horizontal $\C P^1$ with a vertical $\C P^1.$}
\end{figure}

\begin{remark}
 It is almost immediate that Lemma~\ref{lem:ambient_twist_along_vanishing_cycle} has a natural generalization to higher dimensions, where we work with Dehn-Seidel twists instead of
 Dehn twists.
\end{remark}

\subsection{Flexible embedding of surface in $\mathcal{L}(-3)$} \mbox{}

Let $\mathcal{L}(-3)$ denote  the Stein domain corresponding to the complex disk bundle over  
$\C P^1 $ having first  Chern class
$-3.$ Any disk bundle associated to the line bundle
$\mathcal{L}(-3)$ will be denoted by $\mathbb{D}\mathcal{L}(-3).$ The purpose of this section is to
show that given a compact orientable surface of genus $g$
having one puncture there exits a symplectic
embedding of the surface in $\mathcal{L}(-3).$ More precisely, we have the following: 

\begin{lemma}\label{lem:flex_embedding_with_boundary}
Let $\Sigma$ be a once punctured surface of genus 
$g$ with $ g \geq 2.$ There exists a proper symplectic embedding  of $f_{\Sigma}: \Sigma \hookrightarrow \mathcal{L}(-3)$  which satisfies the following properties:
\begin{enumerate}
 \item There exists a plurisubharmonic  proper and exhausting  function $F_{\Sigma}$ on $\mathcal{L}(-3) \rightarrow [0, \infty)$ such that for sufficiently large $M$ $F_{\Sigma}$
has no critical value $c$ with $c \geq M$ and for all
such $c,$ $f_{\Sigma}(\Sigma) \cap F_{\Sigma}^{-1}(M, \infty)$ is a properly embedded annuls in $f(\Sigma).$

\item The embedding is flexible. 
\end{enumerate}

\end{lemma}

\begin{proof}
 We know that $\mathcal{L}(-3)$ admits a Stein Lefschetz fibration having fiber a genus $g$ surface and vanishing
 cycles consisting of curves $a_i, b_j,c_1,$ and
 $c_2$ depicted in Figure~\ref{fig:humphries_generators_with_boundary}.
 Let $\widetilde{f_{\Sigma}}$ be the embedding which send 
 $\widetilde{\Sigma}$, where $\widetilde{\Sigma}$ is a 
 surface with one boundary component, to a smooth fiber of this Lefschetz fibration. Clearly this embedding is symplectic. By adding an appropriate to the Stein Lefschetz fibration, it is easy to see that there exists
 an embedding $f_{\Sigma}$ of $\Sigma$ in the Stein manifold $\mathcal{L}(-3)$ which satisfies the first
 property in the statement. Hence in order to establish
 the theorem, we just need to show that the embedding is
 flexible. We now establish this.
 
 First of all note that  we can assume that the monodromy associated
 to any of the vanishing cycles is supported in a small neighborhood of the vanishing cycle. Hence by an argument
 similar to the one used in the proof of Lemma~\ref{lem:ambient_twist_along_vanishing_cycle}, we get that every Dehn twist along vanishing cycles $a_i, b_j, c_1,$ and $c_2$ can be conjugated. Since the mapping class group of $\Sigma$ is generated
 by Dehn twist along $a_i, b_j, c_1,$ and $c_2$
 the lemma follows. 
 \end{proof}
 
We would like to remark that Lemma~\ref{lem:flex_embedding_with_boundary} is implicitly proved  in the article~\cite{EL} by J. Etnyre and Y. Lekili. The motivation for studying $\mathcal{L}(-3)$ and flexible surfaces 
in them for them is related to their  construction of universal $5$--manifold
in which all contact $3$--manifolds embed.

\section{Proof of Theorem~\ref{thm:symp_lef_embedding}
}\label{sec:proofs_lfe}

 Let us now establish  Theorem~\ref{thm:symp_lef_embedding}. The proof is
 divided in four steps. In the first step we observe that given a Lefschetz fibration $\pi_M: M \rightarrow \C P^1,$ there is a Lefschetz fibration embedding
 of a small tubular neighborhood $\mathcal{N}$ of any singular fiber containing a unique critical point such
 the embedding restricted to any smooth fiber contained
 in $\mathcal{N}$ is flexible. We denote any such embedding by $L_i,$ where $i$ is indexed over the
 set of critical points of the fibration $\pi_M:M \rightarrow \C P^1.$
 
 In the second step we produce embeddings of
 small tubular neighborhoods of  singular
 fibers using
 the first step in such way that no two embeddings 
 intersect, and any pair of smooth  fibers are 
 ambiently symplectically isotopic in $\C P^1 
 \times \C P^1$. In the third step,  we use flexibility of the  embedded fibers to produce an embedding of
 the Lefschetz fibration $\pi_M: M \rightarrow \C P^1$
 restricted to the inverse image of a disk embedded
 in $\C P^1$ which contains all critical values.

In the final step, the triviality of $\pi_1$ of the group
 of symplectic form preserving diffeomorphisms isotopic
 to the identity of any surface of genus $g, g \geq 2$
 is used to conclude that we have a symplectic  Lefschetz fibration embedding of $M$ in $(P, \omega_{pr})$ as claimed.
 
\begin{proof}[Proof of Theorem~\ref{thm:symp_lef_embedding}]\mbox{}

We are given a Lefschetz fibration $\pi:M \rightarrow
\C P^1.$ Let $x_1, x_2, \cdots, x_n$ be the set of
critical points of the fibration. Let $x_i$ be the critical point of $M$ corresponding to the vanishing cycle $c_i.$ Denote by $F_{x_i}$ the fiber which contains the critical point $x_i$ and by $\mathcal{N}(F_{x_i})$ a  neighborhood of $F_{x_i}$ containing the unique critical
point $x_i.$ The neighborhood $\mathcal{N}(F_{x_i})$ is obtained by taking the inverse image under $\pi_M$ of a small disk $\mathbb{D}_{p_i}$  of
radius $\epsilon(p_i),$ which contains only 
$\pi(x_i) = p_i$ as critical value. The proof is
divided in three steps. The first step deals with
embeddings of $\mathcal{N}(F_{x_i}).$

\vspace{0.3cm}

\noindent\textbf{Step-1:}
\vspace{0.3cm}

In this step we prove that there exists an embedding 
$\phi_i$ of a small tubular neighborhood $\mathcal{N}(F_{x_i})$ of $F_{x_i}, i \in \{1, 2, \cdots, l\}$ in $V \times \Sigma$ such that the following diagram commutes:
 
\begin{center}
\begin{equation}\label{diag:local_embedding}
\begin{tikzcd}
 \mathcal{N}(F_{x_i}) \arrow[hookrightarrow]{r}{\phi_i} \arrow[d, "\pi"] &  \C P^1 \times \C P^1 \times 
 \mathbb{D}_{p_i} \arrow[d, "\pi_2"]  \\
\pi(\mathcal{N}(F_{x_1})) \arrow[hookrightarrow]{r} & \mathbb{D}_{p_i}. 
\end{tikzcd}
\end{equation}
\end{center}
 
According to Lemma~\ref{lem:symp_flex_embed_surface} there exists an embedding $H: \Sigma_g \hookrightarrow 
\C P^1 \times \C P^1$ which is symplectic and flexible.
Recall that this is a $(2,g+1)$--curve embedded in
$\C P^1 \times \C P^1.$ We will denote this curve
by $\Sigma_g.$

Since any two essential cycles on the fiber of 
$\pi$ can be conjugated by an isotopy of $M,$ we can
assume that the fiber of $\pi$ over a point $p,$ where $p$ is a point 
$\C P^1$ which lies on the boundary of 
$\mathbb{D}_{p_1}$ admits a symplectic embedding in $\C P^1 \times \C P^1 \times \{p\}$ such that its image is 
the flexibly embedded $(2, g+1)$-curve in $\C P^1 
\times \C P^1 \times \{p\}$ and the vanishing cycle $c_1$
is mapped a  fixed vanishing cycle $\mathcal{C}$ for the pencil $\pi_{(2,n)}:
\C P^1 \times \C P^1 \rightarrow \C P^1$ whose
smooth fiber is $\Sigma_g.$ Let us denote
the vanishing cycle by $\nu.$  Consider 
the blow-up of $\C P^1 \times \C P^1$ at base locus $B$
corresponding to the pencil of $\C P^1 \times \C P^1$
for which $\Sigma_g$ is a smooth fiber.  Let us denote
the blown-up manifold by 
$\widetilde{\C P^1 \times \C P^1}.$ Let $\pi_B: 
\widetilde{\C P^1 \times \C P^1} \rightarrow \C P^1$
be the corresponding Lefschetz fibration. Let $T:
\widetilde{\C P^1 \times \C P^1} \rightarrow \C P^1 \times \C P^1$ be the blow-up projection.  We know that
in $\widetilde{\C P^1 \times \C P^1}$ there exits a 
diagonal embedding of $\mathcal{N}(F_{x_i})$ in
$\widetilde{\C P^1 \times \C P^1} \times \mathbb{D}_{p_1}
\rightarrow \mathbb{D}_{p_1}$ given by diagonal embedding
consisting of $x \rightarrow (x, \pi_B(x))$ followed by
a translation to ensure that the projected disk via 
$\pi_B$ is mapped to the disk $D_{p_1}.$ Call this
embedding $\phi = (\phi_1, \phi_2).$ Given $\phi$
we can define the required embedding of $\mathcal{N}(F_{x_i})$ in $\C P^1 \times \C P^1 \times \mathbb{D}_{p_1}$ by sending $x$ to $(T(\phi_1(x)), \pi_B(x)).$
Let us denote this embedding by $L_1$

\vspace{0.3cm}

\noindent \textbf{Step-2:}

\vspace{0.3cm}

Following step-1, produce
embeddings $L_i$ of $\mathcal{N}(F_{x_i})$ for each 
$i.$ Observe that the embedding is produced such that 
the vanishing cycle $c_i$ is mapped to the curve 
$\mathcal{C},$ which is a fixed vanishing cycle 
for the pencil $\pi_{(2, g+1)}: \C P^1 \times \C P^1
\rightarrow \C P^1.$  In addition, we have that
any embedding $L_i$ restricted to a smooth fiber is
a flexible embedding embedding of the fiber in $\C P^1
\times \C P^1,$ and fibers $L_i (F_{u_i})$ and 
$L_j(F_{u_j}),$ where $u_i \in \mathbb{D}_{p_i}$ is
a regular value, are symplectically isotopic in 
$\C P^1 \times \C P^1.$

\vspace{0.3cm}

\noindent\textbf{Step-3:}
\vspace{0.3cm}

For each $i$ let $Z_i$ be a point on the boundary of $\mathbb{D}_{p_i},$ and  let $F_{z_i}$ denote the embedded fiber over $Z_i.$  Fix the point $Z = Z_1 \in \C P^1$ which does not lie on 
any of the disk $\mathbb{D}_{p_i}$ when $i \neq 1.$ For each $i$ let 
$\overline{ZZ_i}$ be a smooth embedded path joining 
$Z_i$ to $Z.$ Furthermore, let us assume that 
$\overline{ZZ_i}$ and $\overline{ZZ_j}$ intersect only at $Z$ when $i \neq j.$

Let $L$ be a fixed embedding of $(2, g+1)$--curve
$\Sigma_{g}$ in $\C P^1 \times \C P^1 \times \{Z\}$
obtained by the restriction of the embedding $L_1$ to the fiber over $Z = Z_1.$ The triviality of the fiber bundle $\mathcal{P} \rightarrow \C P^1$ implies that there exist
an embedding $\widetilde{L_i}$ of $\Sigma_g$ in 
$\C P^1 \times \C P^1 \times \{z_i\}$ obtained by
identifying the fiber at $Z_i$ with fiber at $Z$ via
this given trivialization such that images of 
$\Sigma_g$ under embedding coming from $L_i| \pi^{-1}\{z_i\}$ and $\widetilde{L}_i$ coincide. 

Observe that the flexibility of the embedding with
image $\widetilde{L_i}$ in $\C P^1 \times \C P^1 \times
\{z_i\}$ implies that these two embeddings are isotopic.
Hence there exist an embedding of $\Sigma_g \times [0,1]$ along the path $\overline{ZZ_i}$ such that 
$\Sigma_g \times \{0\}$ is the embedding $L_i$ restricted
to $\pi^{-1}(Z_i)$ while $\Sigma \times \{1\}$ is 
embedding $L.$

Observe that till now we have produce an embedding 
of $M \setminus \pi^{-1}(\cup \mathbb{D}_{p_i} \cup \overline{zz_i})$ in the manifold $\mathcal{P}.$ 
Consider a small neighborhood of $\cup \mathbb{D}_{p_i}
\cup \overline{zz_i}$ in $\C P^1.$ It is clear that
this neighborhood is a disk $\mathbb{D}$ in $\C P^1$
and there exits a Lefschetz fibration  embedding of $M \setminus  \pi_M^{-1}(\mathbb{D})$ in $\mathcal{P}.$

\vspace{0.3cm}

\noindent\textbf{Step-4:}
\vspace{0.3cm}

Since $M$ is  a closed manifold. We know that the
product of $\tau_{c_i}$ is the identity. This implies
that $\pi^{-1}(\partial \mathbb{D})$ is an embedding
of $\mathbb{S}^1 \times \Sigma_g$ 

Since the genus 
of smooth fiber is at least $2,$ and since the fundamental group of any fixed symplectic form preserving diffeomorphisms of a surface of genus $g$
is trivial, we can assume that $M$ is obtained by
gluing $M \setminus \pi_M^{-1}(\mathbb{D})$ with 
its complement via the identity map.

It is now clear that the constructed embedding 
can be assumed to agree with the embedding 
of $\Sigma \times \partial \mathbb{D}^2$ given by 
$(x, \theta) \rightarrow (L(x), \theta), \theta 
\in \partial \mathbb{D}.$ Let us denote this 
embedding of $M \setminus \pi_M^{-1}(\mathbb{D})$ by
 $\widetilde{\Psi}.$ Finally, we define the required
 embedding of $M$ by the formula:
 
 \[\Psi(x)=\begin{cases} \widetilde{\Psi}(x),        &\text{if $x \in M \setminus \pi_M^{-1}(\mathbb{D})$;}\\ (L(x),x) &\text{otherwise.}\end{cases}\]
 
 Hence the theorem.

\end{proof}

\subsection{Proof of Theorem~\ref{thm:embed_upto_blowup}}\mbox{}

For the sake of completeness, let us discuss how to prove 
 Theorem~\ref{thm:embed_upto_blowup} assuming Theorem~\ref{thm:symp_lef_embedding}.

\begin{proof}[Proof of Theorem~\ref{thm:embed_upto_blowup}]
Let $(M, \omega)$ be a given symplectic manifold. It follows from Theorem~\ref{thm:Donaldson} that  
after finite number of blow-ups there exists a 
(simplified) Lefschetz fibration on the blown-up manifold
$\widetilde{M}.$ It is clear that the manifold $\widetilde{M}$ satisfies the hypothesis of Theorem~\ref{thm:embed_upto_blowup}. Hence
applying Theorem~\ref{thm:embed_upto_blowup}, we get the required embedding
of $\mathcal{B}M$ in $(\C P^1 \times \C P^1 \times \C P^1,\omega_{pr}).$
\end{proof}

\section{Proof of Theorem~\ref{thm:weinstein_embedding}}.\label{sec:stein_lef_fib_embedding}

In order to establish Theorem~\ref{thm:Weinstein_embedding}, we need the notion of Weinstein Lefschetz fibration and the existence of Weinstein
Lefschetz fibration.  We will follow the treatment discussed in the article~\cite{GP} by E. Giroux and J. Pardon. For this we need 
need the definition of Stein domains and associated Lefschetz fibration.  Readers familiar with Stein domains and symplectic geometry 
associated with them can skip this introductory discussion.

 \subsection{Stein domains and Stein Lefschetz fibrations}\mbox{}

Let us begin by recalling the definition of a Stein domain.
 
 \begin{definition}[Stein domain]\label{def:stein_domain}
  A Stein domain is a compact complex manifold $(V, \partial V)$ with boundary together with a
  smooth  Morse function $\phi$ which satisfies the following:
  \begin{enumerate}
   
   \item $\phi$ is $J$--convex, i.e.,  $-d (d\phi  \circ J) (v, Jv) > 0$ for $v \neq 0,$ where $J$ is the
   almost complex structure associated to the complex structure on $V.$
   \item $\partial V$ is a regular level set for the function $\phi.$
  \end{enumerate}
 \end{definition}
 
 Next, we define Stein Lefschetz fibration.

  \begin{definition}
   Let $V$ be a smooth manifolds with corners and let $\pi:V \rightarrow \mathbb{D}^2$ be Lefschetz fibration.
   We say that this Lefschetz fibration is a Stein Lefschetz fibration provided the following properties
   are satisfied:
\begin{enumerate}
  \item There exists a complex structure on $V$ such that the map $\pi$ is holomorphic.
  \item There exists a $J$--convex function $\phi: V \rightarrow \mathbb{R}$ with $\partial_h V = \{\phi = 0\}.$
 \end{enumerate}
 \end{definition}
 
 The most important property of a Stein Lefschetz fibration on $V$ is that $V$ can be smoothened out to
 obtain a Stein domain which is unique up to a deformation in the following sense. 
 
 Let $V_1$ and $V_2$ denote two Stein domains
 obtained as a result of smoothing of a Stein Lefschetz fibration $\pi: V \rightarrow \mathbb{D}$, then 
 there exits a $1$--parametric family $V_t, t \in [1,2]$  of Stein domains connecting $V_1$ and $V_2.$ Hence,
 form now on by $V^{sm}$ we will mean a Stein domain obtained after smoothing of the total space $V$
 of a Stein Lefschetz fibration $\pi:V \rightarrow \mathbb{D}^2.$
 
 The converse of this was established by Giroux and Pardon~\cite[Theorem:1.5]{GP}.
 
 \begin{theorem}[Theorem:1.5~\cite{GP} ]\label{thm:giroux_pardon_existence}
  Let $V$ be a Stein domain. There exists a Stein Lefschetz fibration $\pi: V' \rightarrow \mathbb{D}^2$
  such that $V'^{sm}$ is a deformation of $V.$ Furthermore, we can always provide a fibration 
  for which (1) the boundary of the fiber is connected, (2) every vanishing cycle is essential, and (3) the genus of any smooth
  fiber is at least $3.$
 \end{theorem}

 It follows from the definition  that the $2$--form $-d (d \phi \circ J)$ is a symplectic form on 
 $(V, \partial V).$ Hence, Stein domains are natural examples of exact symplectic manifold. 
 the symplectic geometry of Stein domains is naturally captured by the associated Weinstein domain structure. We now discuss this.
 
 \subsection{Weinstein domains and Weinstein Lefschetz fibration}\mbox{}

 \begin{definition}[Weinstein domain]\label{def:weinstein_domain}
 A Weinstein domain consists of is a compact exact symplectic manifold
 with boundary $(W, \omega = d \lambda),$ where $\lambda$ is 
 a $1$-form, admitting a Morse function $\phi: W \rightarrow 
 \mathbb{R}^{+}$ which satisfy the following:
 \begin{enumerate}
  \item $\partial W$ is a level set of a regular value 
  of the function $\phi.$
  
 \item There exists a vector-field $X_{\lambda}$ defined 
 as $i_{X_{\lambda}} \omega = \lambda$ which is gradient 
 like for the Morse function $\phi.$
 \end{enumerate}
 
 \end{definition}

 \begin{definition}[Weinstein Lefschetz fibrations]\label{def:weinstein_fibration}
Let $(V,J)$ be a smooth manifolds with corners and let $\pi:V \rightarrow \mathbb{D}^2$ be Lefschetz fibration.
   We say that this Lefschetz fibration is a Weinstein Lefschetz fibration provided the following properties are satisfied:
   
   \begin{enumerate}
       \item there exists a $J$--convex function  $\phi:V\rightarrow \mathbb{R}$ such that $\partial_{h}V = \{ \phi = 0\},$ where
       $\partial_h V = \cup_{p \in \mathbb{D}^2} \partial(\pi^{-1}|[p\}),$
       
    \item If $p \in \mathbb{D}^2$ is a regular value then on $\pi^{-1}\{p\}$ the pair $(\phi,J)$ induces a Weinstein structure
    on $\pi^{-1}\{p\}$.

\end{enumerate}
 \end{definition}

 It follows from Theorem~\ref{thm:giroux_pardon_existence} and \cite[Theorem 1.10]{GP} that every $4$--dimensional
 Weinstein  domain admits  a Weinstein Lefschetz fibration as defined in Definition~\ref{def:weinstein_fibration}.

 We now define the universal Weinstein manifold $\mathcal{U}$
 mentioned in the statement of Theorem~\ref{thm:weinstein_embedding}.

\begin{definition}
 [Definition and basic properties of the universal Weinstein $3$--manifold $\mathcal{U}$]
\label{def:universal_stein_3-fold}
  \mbox{}
  \begin{enumerate}
  \item
 Let $\mathcal{L}(-3)$ denote   the complex disk bundle of $-3$ over $S^2.$ It is well known that this bundle
 admits a natural Stein structure. See~\cite{EL} and \cite{OS} for a precise description of this Stein structure.  
 \item
 Let $\mathbb{D}\mathcal{L}(-3)$ denote any disk bundle
 associated to the line bundle $\mathcal{L}(-3).$ 
Observe that $\mathbb{D}\mathcal{L}(-3)$ are Stein domains obtained by considering the inverse images 
under plurisubharmonic Morse functions of the type $(\infty, a]$ which  have no critical
value in the interval $[a, \infty).$
  
\item The product of a Stein manifold with 
  $\mathbb{C}$
  and the product of a Stein domain with the open unit disk is a Stein domain. In particular $\mathbb{D}\mathcal{L}(-3) \times
  \mathbb{D}^2$ is a Stein domain after suitable smoothing of corners,  while 
  $\mathcal{L}(-3) \times \mathbb{C}$ is  Stein manifold. 
  
 \end{enumerate}

 The Weinstein manifold structure associated to the Stein manifold  $\mathcal{L}(-3) \times \mathbb{C}$
with stein structure as discussed above will be called  the universal Weinstein manifold $\mathcal{U}.$

 \end{definition}

Let us now state the result regarding embeddings of Weinstein domains via Lefschetz fibrations.

\begin{theorem}\label{thm:stein_lef_fib_embedding}
 Let $\pi_V:V^4 \rightarrow \mathbb{D}^2$ be any (simplified) Weinstein fibration. There exists a symplectic Lefschetz fibration 
 embedding of $\pi_V:V^4 \rightarrow \mathbb{D}^2$ in the trivial Weinstein fibration 
 $\pi_2:\mathbb{D}\mathcal{L}(-3) \times \mathbb{D}^2 \rightarrow \mathbb{D}^2$  such that $\partial_h V \subset \partial \mathbb{D}\mathcal{L}(-3)
 \times \mathbb{D}^2$ and $\partial_v V \subset \mathbb{D}\mathcal{L}(-3) \times \partial \mathbb{D}^2.$ Furthermore,  by  choosing
 the radius of the disk bundle$\mathbb{D}\mathcal{L}(-3) $ with respect to an auxiliary metric correctly  there exist
 an iso-symplectic Weinstein fibration embedding of $V^4$ in  $\mathbb{D}\mathcal{L}(-3) \times \mathbb{D}^2.$ 
\end{theorem}

We know that the mapping class group of a once
punctured surface is generated by product of Dehn
twists along the curves $a_i, b_j,c_1$ and $c_2.$ For a technical reason --which will become clear while 
going through   the proof of Theorem~\ref{thm:stein_lef_fib_embedding} -- we need 
to show that every Dehn twist along a homologically
essential simple closed curve on a surface with
one boundary component is also isotopic to 
the product of Dehn twist along curves $a_1, b_1,
\cdots, a_{g-1}, b_{g-1}, a_g$, $c_1$, and $c_2$
 We establish this in the following:
 \begin{proposition}\label{pro:essential_cycles_as_product_of_essential}
  Let $(\Sigma_g, \partial \Sigma_g)$ be a surface with one boundary component. Let $\phi$ be an element of the mapping class group 
  that brings a non-separating simple closed curve $c$ to 
  $a_i,$ for some $i,$ then there exists an
  element $\widetilde{\psi}$ of the mapping class group
  of $\Sigma_g$ can be expressed as 
  product of Dehn twists  $\tau_{a_i}, \tau_{a_i}^{-1}$
  $\tau_{b_i}, \tau_{b_i}^{-1},$ $\tau_{c_1}, \tau_{c_1}^{-1},$ $\tau_{c_2},$ and $\tau_{c_2}^{-1},$ and 
  which satisfies the property that $\phi(c) = \widetilde{\phi}(c).$
 \end{proposition}
 \begin{proof}

 We know that the mapping class group of a genus $g$ surface
 with one boundary component is generated by Dehn twists
 $\tau_{a_i}, \tau_{a_i}^{-1},$
  $\tau_{b_i}, \tau_{b_i}^{-1},$ $\tau_{c_1}, \tau_{c_1}^{-1},$ $\tau_{c_2},\tau_{c_2}^{-1},$ $\tau_{d},$ and $\tau_{d}^{-1}$
 Since $d$ does not intersect  with $a_i, b_j, c_1$ and $c_2,$ we can assume that when expressing $\phi$ as 
 a product of positive and negative Dehn twists 
 along curves $a_i, b_j,c_1, c_2$ and $d,$ positive and 
 negative Dehn twists along $d$ appear at the very beginning
 of the product. More precisely, $\phi$ can be expressed 
 as:
 
 $$ \tau_{d}^k \cdot \tau_{d}^{-l} \cdot \psi,  $$
 
 \noindent where $\psi$ is a product of positive and negative
 Dehn twists along curves $a_i, b_j, c_1$ and $c_2.$ 
 
 Now define 
 $\widetilde{\phi} = \tau_{d}^l \circ \tau_d^{-k} \circ \phi.$ Since $c$ can be assumed to be disjoint from $d,$
 we get that $\widetilde{\phi} (c) = \phi (c)$ as claimed. 
 
 \end{proof}

\begin{proof}[Proof of Theorem~\ref{thm:stein_lef_fib_embedding}]\mbox{}
We first discuss the case of symplectic embedding
of Stein Lefschetz fibration.
Consider the given (simplified) Stein Lefschetz fibration $\pi:V^4 \rightarrow \mathbb{D}^2.$  Suppose the regular fiber  is a genus $g, g \geq 3,$ surface with
one boundary component.  

Next recall~\cite{OS} that the Stein domain
$\mathbb{D}\mathcal{L}(-3)$ admits a genus $g$ Stein Lefschetz 
fibration $\pi_g: \mathbb{D}\mathcal{L}(-3) \rightarrow \mathbb{D}^2$ for every $g$ having vanishing cycles $c_1,c_2, a_1, b_1, \cdots a_{g-1}, b_{g-1}, a_g.$ 

Fix an identification of a smooth fiber of $\pi:V \rightarrow \mathbb{D}^2$ with a fiber of $\pi_g:\mathbb{D}\mathcal{L}(-3).$ Let us denote by $\Sigma$ the fiber
of $\pi:V \rightarrow \mathbb{D}^2$ and $f$ it's identification with a smooth fiber of $\pi_g.$ Next,
we claim  that we can produce this identification such
that a  vanishing cycle on the fiber of $\pi:V \rightarrow \mathbb{D}$ is mapped to $a_1$ under this identification.

 To begin with pick an identification $\widetilde{f}: \Sigma \rightarrow \pi_g^{-1}(p) \subset \mathbb{D}\mathcal{L}(-3)$
 for some regular value $p$ of $\pi_g.$ Next observe
 that any  vanishing cycle under this identification goes to 
an essential simple closed curve on the fiber of $\pi_g$
as the fibration $\pi:V \rightarrow \mathbb{D}$ is simplified. 
Now, fix a vanishing cycle $c$ on $\Sigma$ and consider 
the essential cycle $\widetilde{f}(c)$ on $\pi_g^{-1}(p).$
Next we note, Proposition~\ref{pro:essential_cycles_as_product_of_essential} implies that there exists a symplectomorphism $\phi:
f(\Sigma) \rightarrow f(\Sigma)$ which send the vanishing
cycle $f(c)$ to $a_1$  which 
is identify when restricted to $\partial f(\Sigma)$ is 
the identity.  Clearly $\phi \circ \widetilde{f}$ is the required identification $f.$

Next let $c_1, c_2, \cdots, c_n$ be critical points,
let $d_i = \pi(c_i)$ be critical values, and   let $\nu_1,\nu_2, \cdots, \nu_n $ be the corresponding
vanishing cycles of the fibration $\pi:V \rightarrow \mathbb{D}.$  It is clear that by an argument similar
to the one used in the proof of~\ref{thm:symp_lef_embedding} we get for each $i, i = 1
$ to $n$,  an embedding $\psi_i$ of $\pi^{-1}(\mathbb{D}_i,$ where $D_i$ is a small disk containing the point $d_i,$ in $\mathcal{D}(E)(-3) \times \mathbb{D}^2$ which satisfy the following:

\begin{enumerate}
 \item Whenever $i \neq j$ $\psi_i (\pi^{-1}(\mathbb{D}_i)) \cap psi_j (\pi^{-1}(\mathbb{D}_j)) = \emptyset,$
 
\item the following diagram Commutes:

\begin{center}
\begin{equation}\label{diag:local_embedding_stein_case}
\begin{tikzcd}
 \pi^{-1}(\mathbb{D}_i) \arrow[hookrightarrow]{r}{\psi_i} \arrow[d, "\pi"] &  \mathbb{D}\mathcal{L}(-3) \times 
 \mathbb{D}_{p_i} \arrow[d, "\pi_2"]  \\
\mathbb{D}_i \arrow[hookrightarrow]{r} & \mathbb{D}. 
\end{tikzcd}
\end{equation}
\end{center}

\item The vanishing cycle $\nu_i$ is mapped to the
vanishing cycle $a_i$ in $\mathbb{D}\mathcal{L}(-3) \times \{t_i\}$
for $\nu_i \in \pi^{-1}(u_i),$ for some $u_i$ in 
$\partial \mathbb{D}_i.$
\end{enumerate}

Next, observe that embeddings $\psi_1$ and $\psi_i$ for
each $i, i = 2, \cdots, n$ gives rise to two embeddings 
of $\Sigma$ in $\mathbb{D}\mathcal{L}(-3)$ via their restrictions
to $\pi^{-1}(\{u_1\})$ and $\pi^{-1}(\{u_i\})$ respectively. Observe that the first embedding is such that the vanishing cycle $\nu_1$ is mapped to $a_1$ while
the second embedding is such that $\nu_i$ is mapped to 
$a_1.$ Observe again that   Proposition~\ref{pro:essential_cycles_as_product_of_essential} and Lemma~\ref{lem:flex_embedding_with_boundary} 
imply  that these two embeddings are isotopic via a family
of symplectomorphims of $\mathbb{D}\mathcal{L}(-3)$ each of 
which is identity when restricted to 
$\partial \mathcal{D}E(-3.)$ This implies we can produce 
an embedding of $V \setminus \pi^{-1}(\mathcal{N}),$
where $\mathcal{N}$ is a regular neighborhood of $D_i$ union with $\overline{u_1 u_i}$ which is fiber preserving. 
Since $\mathcal{N}$ is a disk embedding in $\mathbb{D},$
extending the embedding restricted to $\pi^{-1}(\partial \mathcal{N})$ fiber wise via  identity to $V,$ we get
the required embedding of $V$ in $\mathbb{D}\mathcal{L}(-3) \times \mathbb{D}.$

It remains to show that we can upgrade the symplectic
embedding of iso-symplectic embedding in a trivial Stein
fibration of the form $\mathbb{D}\mathcal{L}(-3) \times
\mathbb{D}^2.$ In light of Theorem~[Theorem:3.1]\cite{Go} due to Gompf, it is sufficient to produce 
an iso-symplectic identification of fiber of the fiber of the fibration $\pi_V: V \rightarrow \mathbb{D}^2$
with the fiber of $\mathbb{D}\mathcal{L}(-3)$ for some
disk bundle $\mathbb{D}\mathcal{L}(-3).$ Furthermore,
notice that any two symplectic form on a compact surface
having same volume are symplectomorphic implies that
we need an identification such that the induced volume via this identification agrees with the given volume of
a smooth fiber of the fibration $\pi_V:V \rightarrow \mathbb{D}.$ But by adjusting the size of the disk bundle with respect to a fixed auxiliary metric,
this is always possible. Hence, we get the
required iso-symplectic embedding.

\end{proof}

\begin{proof}[Proof of Theorem~\ref{thm:weinstein_embedding}]\mbox{}

 Given a Stein $2$--manifold $V$ admitting a plurisubharmonic exhaustive Morse function $v: V \rightarrow \mathbb{R}$
with only finite number of critical points, let 
$t_0 \in \mathbb{R}$ be such that there are no critical 
values that belong to the interval $[t_0, \infty).$
Observe that if we iso-symplectically and properly embed the Stein domain $V^{t_0}
= v^{-1}(-\infty, t_0]$ in the Stein domain $\mathcal{U}^{\beta},$ then we can properly and iso-symplectcially
embed $V$ in $\mathcal{U}.$  Hence
our task is to establish this. 

In order to establish this we first apply Theorem~\ref{thm:giroux_pardon_existence} to produce a Stein Lefschetz fibration of $\pi:V^{t_0} \rightarrow \mathbb{D}^2$ such that the fibers of the Stein fibration are connected. Next, we apply \cite[Proposition:1.5]{LP}
to get a new Stein Lefschetz fibration with connected boundary, connected fiber, and the fibration having  every vanishing cycle essential. That is, we produce a simplified Lefschetz fibration on $V^{t_0}.$

Hence, by Theorem~\ref{thm:stein_lef_fib_embedding} there exits a symplectic Lefschetz fibration embedding of $\pi:V^{t_0} \rightarrow \mathbb{D}$ in the Stein trivial Stein Lefschetz fibration associated to $\mathcal{U}^{\beta}$ as $\mathcal{U}$ is just $\mathcal{L}(-3)
\times \mathbb{C}.$

Finally, observe that applying 
the smoothing of corner operation on the Stein Lefschetz
fibration associated to $\mathcal{U}^{\beta}$ -- due to the facts that (1) the smoothing
operation induces a smoothing operation on $V,$ and (2) the smooth is operation
is canonical up to deformation equivalence -- we get the required symplectic embedding of $V^{t_0}$ in $\mathcal{U}^{\beta}.$

\end{proof}

In the end, we would like to remark that if $W$ is a Weinstein manifold such  that $W$ is obtained from 
the Stein domain $\mathcal{U}^{\beta} \coloneqq f^{-1}\left(\left(-\infty, \beta\right]\right)
\subset \mathcal{U}$  by
attaching finite number of  Weinstein  $2$-handles along the contact boundary of $\mathcal{U}^{\beta}$ as described in~\cite{CE}, then $W$ is also universal.   This is because, for the dimension reasons,  the core of any attached  Weinstein handle can always be 
assumed to be disjoint form $\partial V^{t_0}$ embedded in $\partial 
\mathcal{U}^{\beta}.$

\bibliographystyle{amsplain}

\begin{thebibliography}{100}



\bibitem{AMP04} Amor\'os, Jaume; Mu\~noz, Vicente; Presas, Francisco,
\textit{Generic behavior of asymptotically holomorphic Lefschetz pencils.} 
J. Symplectic Geom. 2 (2004), no. 3, 377–-392. 




\bibitem{Ar} D. Auroux, \textit{Asymptotically holomorphic families of symplectic submanifolds,} Geom. Funct. Anal., Vol. 7, (1997) 971 -- 995.




\bibitem{AK} D. Auroux and L. Katzarkov, 
\textit{A degree doubling formula for braid monodromy and 
Lefschetz pencil}, Pure and applied mathematics quarterly,
Vol. 4, (2008) 237--318.




\bibitem{BO} I. Baykur and S. Osamu,  \textit{Simplified broken Lefschetz fibrations and trisections of 4-manifolds,}  Proc. Natl. Acad. Sci. USA 115,  Vol. 43 (2018), 
10894 -- 10900.






\bibitem{BM} F. Benson and D. Margalit, \textit{A primer on mapping class groups},  Princeton Mathematical series, Vol. 49, Princeton University Press, (2012).
 



\bibitem{CE} K. Ceiliback and Y. Eliashberg, \textit{From Stein to Weinstein and back}, Colloquium Publication,
Vol .59. AMS.





\bibitem{De} M. Dehn, \textit{Die Gruppe der Abbildungsklassen.}
 Acta Math., 69(1), (1938), 135--206.

\bibitem{Do1} S. K. Donaldson, \textit{Symplectic
submanifolds and almost--complex geometry}. Journal
of Differential Geometry, Vol. 44, (1996), 666--705.


\bibitem{Do} S. K. Donaldson, \textit{ Lefschetz   pencils on symplectic manifolds}.
Journal of Differential geometry, Vol. 55 (1999), 205--236.


\bibitem{EL} J. Etnyre and Y. Lekili, \textit{Embedding of 
all contact $3$--manifolds in a fixed contact $5$--manifold}, Journal of LMS, Vol. 99(1), (2018).

\bibitem{Go} R. Gompf, \textit{Topology of Symplectic manifolds}, Turk. Journal of Math., Vol. 25, (2001),
43--59.

\bibitem{Gr} M. Gromov, \textit{Partial differential relations}. A series of modern surveys in mathematics.
Springer. (1986).


\bibitem{GK} D. Gay and R. Kirby, \textit{Constructing Lefschetz type fibration on
$4$--manifolds}, Geometry and Topology, Vol. 11(4), (2005), 2075--2115.

\bibitem{GPan} A. Ghanwat and D. M. Pancholi,\textit{Embeddings of $4$--manifolds in $\mathbb{C}P^3.$}
 arXiv:2002.11299v2 [math.GT] 


\bibitem{GP} E. Giroux and J. Pardon, \textit{Existence of Lefschetz fibrations on Stein and 
Weinstein domains}, Geometry and Topology, Vol. 21(2), 963--997.

\bibitem{GS} R. Gompf and A. I. Stipsicz, \textit{$4$--manifolds and Kirby calculus.}  Graduate Studies in Mathematics, Vol. 20,  American Mathematical Society, Providence, RI, (1999).




\bibitem{Hu} Humpries, \textit{Generators of the mapping class group}. Topology of Low-dimensional
manifolds, Proceedings of the Second Sussex Conference, (1977), 44--47.



%


\bibitem{HY}  S. Hirose and  A. Yasuhara, \textit{Surfaces in 4-manifolds and their mapping class groups.}  Topology, Vol. 47(1) (2008), 41--50.

\bibitem{Ko} K. Kodaria, \textit{On K\"ahler varieties of restricted type (an intrinsic characterization of algebraic varieties)}, Annals of Mathematics, Vol. 60  (1), 28--48.

\bibitem{Li} W.B.R  Lickorish,  \textit{A representation of orientable 
 combinatorial $3$--manifolds}.  Ann. of Math., Vol. 76(2), (1962),  531--540.

\bibitem{Li1} W.B.R Lickorish, \textit{A finite set of generators for the homeotopy
group of a $2$--manifold.} Proc. Cambridge Philos. Soc., Vol.60, (1964), 769--778.

\bibitem{LP} A. Loi and R. Piergallini, \textit{Compact Stein surfaces with boundary as branched covers of $B^4$}, Invent. Math. Vol.143 (2), (2001), 325--348.


\bibitem{Na} J. Nash, The Imbedding Problem for Riemannian Manifolds,
{\it Annals of Mathematics}, Vol. 63 (1), 20--63,  (1956).



%

\bibitem{OS} B.  Ozbagci and  A.  I.  Stipsicz, \textit{Surgery  on  contact  
$3$--manifolds  and  Stein  surfaces,}  Vol. 13  of Bolyai  Society Mathematical Studies. Springer-Verlag, Berlin, 2004.



\bibitem{Wh} H. Whitney, \textit{The self-intersections of a smooth $n$--manifold in $2n$--space,}  Ann. of Math,
vol. 45(2), (1944), 200--246.



%
%
%
%
%
%




\end{thebibliography}

\end{document}